\documentclass[12pt]{article}

\usepackage[utf8]{inputenc}

\usepackage[a4paper,nomarginpar]{geometry}
\usepackage[english]{babel}

\usepackage{amsmath,amsfonts,amssymb,amsthm}
\usepackage{amscd}

\newtheorem{theorem}{Theorem}

\newtheorem{corollary}[theorem]{Corollary}

\theoremstyle{definition}

\newtheorem{example}[theorem]{Example}

\theoremstyle{remark}
\newtheorem{remark}[theorem]{Remark}

\newcommand{\N}{\mathbb{N}} 
\newcommand{\Z}{\mathbb{Z}} 
\newcommand{\T}{\mathbb{T}} 
\newcommand{\K}{\mathbb{K}} 

\newcommand{\orb}{\operatorname{Orb}}

\begin{document}

\title{Strong mixing measures for linear operators and frequent hypercyclicity}


\author{M. Murillo-Arcila and  A. Peris\footnote{IUMPA, Universitat Polit\`{e}cnica de Val\`{e}ncia, Departament de Matem\`{a}tica Aplicada, Edifici 7A, 46022 Val\`{e}ncia, Spain. e-mail: aperis@mat.upv.es}}

\date{ }

\maketitle

\begin{abstract}
We construct strongly mixing  invariant measures with full support
for operators  on $F$-spaces which satisfy the Frequent Hypercyclicity Criterion.
For unilateral backward shifts on sequence spaces, a slight modification shows
that one can even obtain exact invariant measures.
\end{abstract}

\section{Introduction}

 We recall that an operator $T$ on a topological vector space $X$ is called \emph{hypercyclic} if there is a vector $x$  in $X$  such that its \emph{orbit} $\orb (x,T)=\{x, Tx, T^2x,\dots\}$ is dense in $X$.   The recent books \cite{bayart_matheron2009dynamics} and \cite{grosse-erdmann_peris2011linear} contain the theory and most of the recent advances on hypercyclicity and linear dynamics, especially in topological dynamics.

Here we are concerned with measure theoretic  properties. Let  $(X,\mathfrak{B},\mu)$ be a probability space, where $X$ is a topological space and $\mathfrak{B}$ denotes the $\sigma$-algebra of Borel subsets of $X$. We will say that a Borel probability measure $\mu$ has \emph{full support} if for each non-empty open set $U\subset X$ we have $\mu(U)>0$. A measurable map $T:(X,\mathfrak{B},\mu)\rightarrow(X,\mathfrak{B},\mu)$ is called a \emph{measure-preserving} transformation  if $\mu(T^{-1}(A))=\mu(A)$ for all $A\in\mathfrak{B})$.  $T$ is said to be \emph{strongly mixing} with respect to $\mu$ if
$$
\lim_{n\rightarrow\infty}\mu(A\cap T^{-n}(B))=\mu(A)\mu(B)\qquad (A,B\in\mathfrak{B}),
$$
and it is \emph{exact} if given $A\in\bigcap_{n=0}^\infty T^{-n}\mathfrak{B}$ then either $\mu(A)=0$ or $\mu(A)=1$.  The interested reader is referred  to \cite{walters-peter1982an,einsiedler_ward2011ergodic} for a detailed account on the above properties.

Ergodic theory was first used for the dynamics of linear operators by  Rudnicki \cite{Rdn93} and Flytzanis \cite{Fly95}. During the last few years it has been given special attention thanks to the work of Bayart and Grivaux \cite{BaGr05,BaGr06}. The papers  \cite{BadGr07a,BaGr07,bayart_matheron0000mixing,DFGP12,Gri,Rud12}, for instance, contain recent advances on the subject.

The concept of frequent hypercyclicity was introduced by Bayart and Grivaux \cite{BaGr06} inspired by Birkhoff's ergodic theorem.
They also gave the  first version of a Frequent Hypercyclicity Criterion, although we will consider the formulation of Bonilla and Grosse-Erdmann \cite{bonilla_grosse-erdmann2007frequently}  for operators on separable $F$-spaces. Another (probabilistic) version of it was given by Grivaux \cite{Gri06}.

 We derive under the hypothesis of Bonilla and Grosse-Erdmann   a stronger result by showing that a $T$-invariant mixing measure can be obtained. Recently, Bayart and Matheron gave very general conditions expressed on eigenvector fields associated with unimodular eigenvalues under which an operator $T$ admits a $T$-invariant mixing
 measure \cite{bayart_matheron0000mixing}. Actually, on the one hand our results can be deduced from \cite{bayart_matheron0000mixing} in the context of complex Fr\'{e}chet  spaces, and on the other hand we only need rather elementary tools.

From now on, $T$ will be an operator defined on a separable $F$-space $X$.

\section{Invariant measures and the frequent hypercyclicity criterion}

We recall that a series $\sum_nx_n$ in $X$ \emph{converges unconditionally} if it converges and, for any $0$-neighbourhood $U$ in $X$, there
exists some $N\in \N$ such that $\sum_{n\in F} x_n\in U$ for every finite set $F\subset \{ N, N+1,N+2,\dots \}$.

We are now ready to present our main result. The idea behind the proof is to construct a ``model'' probability space $(Z,\overline{\mu})$ and a (Borel) measurable map $\Phi:Z\to X$, where
$Z\subset \N^\Z$ is such that $\sigma (Z)=Z$ for the Bernoulli shift $\sigma (\dots , n_{-1},n_0,n_1,\dots )=(\dots , n_{0},n_1,n_2,\dots )$, $\overline{\mu}$ is a $\sigma^{-1}$-invariant strongly mixing measure, $Y:=\Phi (Z)$ is a $T$-invariant dense subset of $X$, $\Phi \sigma^{-1}=T\Phi$, and then the Borel probability measure $\mu$ on $X$ defined by $\mu(A)=\overline{\mu}(\Phi^{-1}(A))$, $A\in\mathfrak{B}(X)$, is $T$-invariant and strongly mixing. We will use the slight generalization of the Frequent Hypercyclicity Criterion for operators given in
\cite[Remark 9.10]{grosse-erdmann_peris2011linear}.

\begin{theorem}\label{fhc}
Let $T$ be an operator on a separable $F$-space $X$. If there is a dense subset $X_0$ of $X$ and a sequence of maps $S_n:X_0\rightarrow X$ such that, for each $x\in X_0$,
\begin{itemize}
\item[(i)]$\sum_{n=0}^\infty T^nx$ converges unconditionally,
\item[(ii)]$\sum_{n=0}^\infty S_nx$ converges unconditionally, and
\item[(iii)]$T^nS_nx=x$ and $T^mS_nx=S_{n-m}x$ if $n>m$,
\end{itemize}
then there is a $T$-invariant strongly mixing Borel probability measure $\mu$ on $X$ with full support.
\end{theorem}

\begin{proof}
We suppose that $X_0=\{x_n \ ; \ n\in\mathbb{N}\}$ with $x_1=0$ and $S_n0=0$ for all $n\in\N$. Let $(U_n)_n$ be a basis of balanced open $0$-neighbourhoods in $X$ such that
$U_{n+1}+U_{n+1} \subset U_n$, $n\in \mathbb{N}$.
By $(i)$ and $(ii)$, there exists an increasing sequence of positive integers $(N_n)_n$  with $N_{n+2}-N_{n+1}> N_{n+1}-N_n$ for all
$n\in \N$ such that
\begin{equation}\label{unconditional}
\sum_{k>N_n}T^kx_{m_k}\in  U_{n+1} \mbox{ and } \sum_{k>N_n}S_kx_{m_k}\in U_{n+1}, \mbox{ if } m_k\leq 2l, \mbox{ for } N_l< k\leq N_{l+1}, \ l\geq n.
\end{equation}
Actually, this is a consequence of the completeness of $X$ and the fact that, for each $0$-neighbourhood $U$ and for all $l\in\N$, there is $N\in \N$ such that $\sum_{k\in F}T^kx\in U$ and $\sum_{k\in F}S_kx\in U$ for any finite subset $F\subset ]N,+\infty [$ and for each $x\in\{ x_1,\dots ,x_{2l}\}$.

\vspace{3mm}

\noindent \textbf{1.-The model probability space $(Z,\overline{\mu})$.}

\vspace{3mm}

We define $K=\prod_{k\in\mathbb{Z}}F_k$ where
$$
F_k=\{1,\ldots m\} \mbox{ if } N_m<|k|\leq N_{m+1}, \ m\in\N, \mbox{ and } F_k=\{1\}, \mbox{ if } |k|\leq N_1.
$$
Let $K(s):=\sigma^s(K)$, $s\in\Z$, where $\sigma :\N^\Z\to\N^\Z$ is the backward shift. $K(s)$ is a compact space when endowed with the product topology inherited from $\mathbb{N}^\mathbb{Z}$, $s\in \Z$.

We consider in $\mathbb{N}^\mathbb{Z}$ the product measure $\overline{\mu}=\bigotimes_{k\in\mathbb{Z}}\overline{\mu_k}$, where $\overline{\mu_k}(\{n\})=p_n$ for all $n\in\mathbb{N}$ and $\overline{\mu_k}(\mathbb{N})=\sum_{n=1}^\infty p_n=1$, $k\in\Z$. The values of  $p_n\in ]0,1[$ are
 selected such that,  if
 $$
 \beta_j:=\left(\sum_{i=1}^{j}p_i\right)^{N_{j+1}-N_{j}}, \ j\in\N, \mbox{ then } \prod_{j=1}^\infty\beta_j>0.
 $$
Let $Z=\bigcup_{s\in\Z} K(s)$. We have
$$
\overline{\mu}(Z)\geq \overline{\mu}(K) = \prod_{|k|\leq N_1}\overline{\mu}_k(\{1\}) \prod_{l= 1}^\infty \left( \prod_{N_l<|k|\leq N_{l+1}}
 \overline{\mu}_k(\{1,\dots,l\})\right)=p_1^{2N_1+1}(\prod_{l=1}^\infty\beta_l)^2>0.
$$
It is well-known \cite{walters-peter1982an} that $\overline{\mu}$ is a $\sigma^{-1}$-invariant strongly mixing Borel probability  measure. Since $\sigma (Z)=Z$ and it has positive measure, then $\overline{\mu}(Z)=1$.

\vspace{3mm}

\noindent \textbf{2.-The map $\Phi$.}

\vspace{3mm}

Given $s\in \Z$ we define the  map $\Phi:K(s)\rightarrow X$  by
\begin{equation}\label{themap}
\Phi((n_k)_{k\in\mathbb{Z}})=\sum_{k<0}S_{-k}x_{n_k}+x_{n_0}+\sum_{k>0}T^kx_{n_k}.
\end{equation}
$\Phi$ is well-defined since, given $(n_k)_{k\in\Z}\in K(s)$ and for $l\geq |s|$, we have $n_k\leq 2l$ if $N_l<|k|\leq N_{l+1}$, which shows
the convergence of the series in \eqref{themap} by \eqref{unconditional}.

$\Phi$ is also continuous. Indeed, let
$(\alpha (j))_j$  be a sequence of elements of $K(s)$ that converges to $\alpha\in K(s)$ and fix any $n\in\N$ with $n>|s|$.
We will find $n_0\in\N$ such that $\Phi(\alpha(j))-\Phi(\alpha)\in U_n$ for $n\geq n_0$. To do this, by definition of the topology in $K(s)$ there
exists $n_0\in\N$ such that
$$
\alpha (j)_k=\alpha_k\quad\mbox{if}\quad |k|\leq N_{n+1} \mbox{ and } j\geq n_0.
$$
By \eqref{unconditional} we have
$$
\Phi(\alpha (j))-\Phi(\alpha)=\sum_{k<-N_{n+1}}S_{-k}(x_{\alpha (j)_k}-x_{\alpha_k})+\sum_{k> N_{n+1}}T^k(x_{\alpha (j)_k}-x_{\alpha_k})\in U_{n}
$$
for all $j\geq n_0$. This shows the continuity of $\Phi:K(s)\to X$ for every $s\in\Z$.

The map $\Phi$ is then well-defined on $Z$, and $\Phi:Z\to X$ is measurable (i.e., $\Phi^{-1}(A)\in\mathfrak{B}(Z)$ for every $A\in\mathfrak{B}(X)$).

\vspace{3mm}

\noindent \textbf{3.-The measure $\mu$ on $X$.}

\vspace{3mm}

 $L(s):=\Phi(K(s))$ is compact in $X$, $s\in\Z$, and $Y:=\bigcup_{s\in\Z} L(s)$ is  a   $T$-invariant Borel subset of $X$ because $\Phi \sigma^{-1}=T\Phi$.

We then define in $X$ the  measure $\mu(A)=\overline{\mu}(\Phi^{-1}(A))$ for all $A\in\mathfrak{B}(X)$. Obviously, $\mu$ is well-defined and it is a $T$-invariant strongly mixing Borel probability  measure.
The proof is completed by showing that $\mu$ has full support. Given a non-empty open set $U$ in $X$, we pick $n\in\N$ satisfying $x_n+U_n\subset U$.
Thus
$$
\mu(U)\geq \mu(\{x=x_n + \sum_{k>N_n}T^kx_{m_k}+ \sum_{k>N_n}S_kx_{m_k} \ ; \ m_k\leq 2l \mbox{ for }  N_l< k\leq N_{l+1}, \ l\geq n \})
$$
$$
\geq \overline{\mu}_0(\{ n\})\prod_{0<|k|\leq N_n}\overline{\mu}_k(\{1\}) \prod_{l= n}^\infty \left( \prod_{N_l<|k|\leq N_{l+1}}
 \overline{\mu}_k(\{1,\dots,2l\})\right)>p_n p_1^{2N_n}(\prod_{l=n}^\infty\beta_l)^2>0.
$$
\end{proof}

As we mentioned in the Introduction, Theorem~\ref{fhc} can be deduced from \cite[Corollary 1.3]{bayart_matheron0000mixing} when dealing with operators on separable complex Fr\'{e}chet spaces. Indeed, the argument of \'{E}. Matheron is the following (we thank S. Grivaux for letting us know about it):

Let $T:X\to X$ be an operator on a separable complex Fr\'{e}chet space $X$ satisfying the hypothesis of the Frequent Hypercyclicity Criterion given in Theorem~\ref{fhc}, and suppose that $X_0=\{ x_n \ ; \ n\in\N\}$. We define the following family of continuous $\T$-eigenvector fields for $T$
$$
E_m(\lambda)=\sum_{n\geq 0}\lambda^{-n}T^n x_m+\sum_{n\in \N} \lambda^nS_nx_m, \ \ \lambda \in\T, \ \ m\in \N .
$$
They span $X$ since, for any functional  $x^*$ that vanishes on $E_m(\lambda)$ for each $\lambda\in\T$ and $m\in\N$, the equality $\langle x^* , E_m(\lambda)\rangle=0$ for fixed $m$ and for all $\lambda \in\T$ implies  that  $\langle x^*, T^n x_m\rangle=0$ for every $n\geq 0$. Thus,  $\langle x^*, x_m\rangle=0$ for each $m\in \N$, and by density $x^*=0$.

The previous Theorem can be applied to different classes of operators. A distinguished one is the class of weighted shifts on sequence $F$-spaces.

By a \emph{sequence space} we mean a topological vector space $X$ which is continuously included in $\omega$, the countable product of the scalar field $\K$. A \emph{sequence $F$-space} is a sequence space that is also an $F$-space. Given a sequence $w=(w_n)_n$ of non-zero weights, the associated  \emph{unilateral} (respectively, \emph{bilateral}) \emph{weighted backward shift} $B_w:\K^\N\to\K^\N$  is defined by $B_w(x_1,x_2,\dots )=(w_2x_2,w_3x_3,\dots )$ (respectively, $B_w:\K^\Z\to\K^\Z$ is defined by $B_w(\dots,x_{-1},x_0,x_1,\dots )=(\dots,w_0x_{0},w_1x_1,w_2x_2,\dots )$). When a sequence $F$-space $X$ is invariant under certain weighted backward shift $T$, then $T$ is also continuous on $X$ by the closed graph theorem.

 We refer the reader to, e.g., Chapter 4 of \cite{grosse-erdmann_peris2011linear} for more details about hypercyclic and chaotic weighted shifts on  Fr\'{e}chet sequence spaces. In particular, Theorems 4.6 and 4.12 in \cite{grosse-erdmann_peris2011linear} (we refer the reader to \cite{grosse-erdmann2000hypercyclic}
 for the original results) remain valid for $F$-spaces, and a bilateral (respectively, unilateral) weighted backward shift $T:X\to X$ on a sequence $F$-space $X$ in which the canonical unit vectors $(e_n)_{n\in\Z}$ (respectively, $(e_n)_{n\in\N}$) form an unconditional basis is chaotic if, and only if, $\sum_{n\in\Z} e_n$ (respectively, $\sum_{n\in\N} e_n$) converges unconditionally.

\begin{corollary}
Let $T:X\rightarrow X$ be a chaotic bilateral weighted backward shift on a sequence $F$-space $X$ in which $(e_n)_{n\in\Z}$ is an unconditional basis. Then there exists a $T$-invariant strongly mixing Borel probability measure on $X$ with full support.
\end{corollary}

\begin{remark}
The preceding result can be improved if $T$ is a unilateral backward shift operator on a sequence $F$-space. In that case, there exists a $T$-invariant exact Borel probability measure on $X$ with full support.
\end{remark}

\begin{proof}
Let $M=\{z_n \ ; \ n\in\N\}$ be a countable dense set in $\mathbb{K}$  with $z_1=0$. Let $(U_n)_n$ be a basis of balanced open $0$-neighbourhoods in $X$ such that
$U_{n+1}+U_{n+1} \subset U_n$, $n\in \mathbb{N}$. Since $T$ is chaotic, $\sum_{n=1}^\infty e_n$ converges unconditionally, so there exists an increasing sequence of positive integers $(N_n)_n$  with $N_{n+2}-N_{n+1}> N_{n+1}-N_n$ for all
$n\in \N$ such that
\begin{equation}\label{unconditional2}
\sum_{k>N_n}\alpha_ke_k\in  U_{n+1} , \mbox{ if } \alpha_k \in \{ z_1,\dots ,z_{2m}\}, \mbox{ for } N_m< k\leq N_{m+1}, \ m\geq n.
\end{equation}
We define $K=\prod_{k\in\mathbb{N}}F_k$ where
$$
F_k=\{z_1,\ldots z_m\} \mbox{ if } N_m<k\leq N_{m+1}, \ m\in\N, \mbox{ and } F_k=\{z_1\}, \mbox{ if } k\leq N_1.
$$
Let $K(s):=\sigma^s(K)$, $s\geq 0$. $K(s)$ is a compact space when endowed with the product topology inherited from ${M}^\mathbb{N}$, $s\geq 0$. We consider in ${M}^\mathbb{N}$ the product measure $\overline{\mu}=\bigotimes_{k\in\mathbb{N}}\overline{\mu_k}$, where $\overline{\mu_k}(\{z_n\})=p_n$ for all $n\in\mathbb{N}$ and $\overline{\mu_k}(M)=\sum_{n=1}^\infty p_n=1$, $k\in\N$. As before,
    we select the sequence $(p_n)_n$ of positive numbers such that, if
     $$
     \beta_j=\left(\sum_{i=1}^{j}p_i\right)^{N_{j+1}-N_{j}}, \mbox{ then } \prod_{j=1}^\infty\beta_j>0.
     $$
    It is known \cite[\S 4.12]{walters-peter1982an} that $\overline{\mu}$ is a $\sigma$-invariant exact Borel probability  measure. By setting $Z=\bigcup_{s\geq 0} K(s)$, we have $\overline{\mu}(Z)=1$.

    Now we define the  map $\Phi:K(s)\rightarrow X$ given by
    $$
    \Phi((\alpha_k)_{k\in\mathbb{N}})=\sum_{k=1}^\infty\alpha_ke_k.
    $$
    $\Phi$ is (well-defined and) continuous, $s\geq 0$. We have that $\Phi: Z\to X$ is measurable.
    $L(s):=\Phi(K(s))$ is compact in $X$, $s\geq 0$, and $Y:=\bigcup_{s\geq 0} L(s)=\Phi (Z)$ is  a   $T$-invariant Borel subset of $X$.

    We then define on $X$ the  measure $\mu(A)=\overline{\mu}(\Phi^{-1}(A))$ for all $A\in\mathfrak{B}(X)$. As in Theorem~\ref{fhc}, we conclude that $\mu$ is well-defined, and now it is a $T$-invariant exact Borel probability  measure with full support.
    \end{proof}

    Devaney chaos is therefore a sufficient condition for the existence of strongly mixing measures within the framework of weighted shift operators on sequence $F$-spaces. For some natural spaces it is even a characterization of this fact. For instance, F. Bayart and I. Z. Ruzsa \cite{BR} recently proved that weighted shift operators on $\ell^p$, $1\leq p<\infty$, are frequently hypercyclic if, and only if, they are Devaney chaotic. It turns out that this is equivalent to the existence of an invariant strongly mixing Borel probability measure with full support on $\ell^p$. Also, for the space $\omega$, every weighted shift operator is chaotic \cite{grosse-erdmann2000hypercyclic}. In particular, for the unilateral case we obtain exact measures.

    \begin{example}
    Every unilateral weighted backward shift operator on $\omega=\K^\N$ admits an invariant exact Borel probability measure with full support on $\omega$.
    \end{example}

    We finish the paper by mentioning that a continuous-time version of Theorem \ref{fhc} can be given  by using the Frequent Hypercyclicity Criterion for $C_0$-semigroups introduced in \cite{mangino_peris2011frequently}. This is part of a forthcoming paper  \cite{murillo_peris0000strong}.

    \section*{Acknowledgements}

    This work was supported in part by MEC and FEDER, Project
MTM2010-14909, and by GV, Project PROMETEO/2008/101. The first author was also supported by a
grant  from the FPU Program of MEC.  We thank the referee whose detailed report led to an improvement
in the presentation of this work.

    \end{document}